\documentclass[reqno,a4paper,11pt]{amsart}
\title{A Minkowski theorem for Meyer sets}
\date{\today}
\usepackage[latin1]{inputenc}
\usepackage[english]{babel}
\usepackage[T1]{fontenc}
\usepackage{amsfonts}
\usepackage{amsmath}
\usepackage{amssymb}
\usepackage{stmaryrd}
\usepackage{amsthm}
\usepackage{enumerate}
\usepackage{pgf,tikz}
\usetikzlibrary{arrows}
\usepackage{float}
\usepackage{caption}
\usepackage{dsfont}

\author{Pierre-Antoine Guih\'eneuf}
\address[Pierre-Antoine Guih\'eneuf]{Universidade Federal Fluminense, Instituto de Matemática e Estat\'istica, Rua Mário Santos Braga S/N, 24020-140 Niteroi, RJ, Brasil}
\email{pguiheneuf@id.uff.br}

\author{\'Emilien Joly}
\address[\'Emilien Joly]{Modal'X, Bureau E08, Bâtiment G,
Université Paris Ouest,
200 avenue de la République,
92000 Nanterre}
\email{emilien.joly@u-paris10.fr}

\newtheorem{lemme}{Lemma}
\newtheorem{theoreme}[lemme]{Theorem}
\newtheorem{prop}[lemme]{Proposition}
\newtheorem{coro}[lemme]{Corollary}

\newtheorem*{theorem}{Theorem}

\theoremstyle{definition}
\newtheorem{definition}[lemme]{Definition}

\theoremstyle{remark}
\newtheorem{rem}[lemme]{Remark}

\newtheorem{ex}[lemme]{Example}

\newcommand{\BS}{\overline{B}_R^S}

\newcommand{\N}{\mathbf{N}}

\newcommand{\R}{\mathbf{R}}

\newcommand{\Q}{\mathbf{Q}}
\newcommand{\Z}{\mathbf{Z}}
\newcommand{\varep}{\varepsilon}

\newcommand{\Leb}{\mathrm{Vol}}

\newcommand{\card}{\#}

\newcommand{\1}{\mathbf 1}
\newcommand{\DG}{\Delta\Gamma}

\begin{document}

\begin{abstract}
In this paper, we generalize Minkowski's theorem. This theorem is usually stated for a centrally symmetric convex body and a lattice both included in $\R^n$. In some situations, one may replace the lattice by a more general set for which a notion of density exists. In this paper, we prove a Minkowski theorem for Meyer sets, which bounds from below the frequency of differences appearing in the Meyer set and belonging to a centrally symmetric convex body. In the later part of the paper, we develop quite natural applications of this theorem to Diophantine approximation and to discretization of linear maps.
\end{abstract}


\maketitle

\section*{Introduction}

Minkowski theorem states that if a convex subset of $\R^n$ is centrally symmetric with respect to 0 and has a big enough volume, then it contains a non-trivial point with integer coordinates, i.e. a point of $\Z^n$. This result was proved by H. Minkowski in 1889, and initiated a whole field, now called \emph{geometry of numbers} (see for example the books \cite{minkowski1910geometrie}, \cite{MR893813}, \cite{MR1020761} or \cite{MR1434478}). Since then, this theorem has led to many applications in various fields such as algebraic number theory, Diophantine approximation, harmonic analysis or complexity theory. 

The goal of the present paper is to state a Minkowski theorem in the more general context where the lattice $\Z^n$ is replaced by a Meyer subset of $\R^n$. A set $\Gamma\subset\R^n$ is \emph{Meyer} if both $\Gamma$ and its set of differences $\Delta\Gamma$ are uniformly discrete and relatively dense (see Definition~\ref{MeyerSet}). In particular, this definition implies that the uniform upper density $D^+(\Gamma)$ of $\Gamma$ (see Definition~\ref{uud}) is positive and finite. Given a Meyer set $\Gamma\subset \R^n$ and a centrally symmetric convex body $S$, it is always possible to remove a finite subset from $\Gamma$ such that the resulting set $\Gamma'$ is still Meyer and satisfies $\Gamma'\cap S = \emptyset$. Therefore, one cannot hope to get a meaningful statement of Minkowski theorem involving only the number of points in $S\cap\Gamma$ and $D^+(\Gamma)$. The solution is to average upon the whole Meyer set, and to introduce the so-called \emph{frequency of differences}. The frequency of the difference $u\in\Z^n$ is defined as the density $\rho_\Gamma(u)$ of the set $\Gamma\cap(\Gamma -u)$ over the density of $\Gamma$ (Definition~\ref{DefDiff}). Again, the fact that the set $\Gamma$ is Meyer is important here: it implies that the support of $\rho_{\Gamma}$ is uniformly discrete. The main result of this paper is the following (Theorem~\ref{MinkAlm}).

\begin{theorem}
Let $\Gamma\subset \R^n$ be a Meyer set, and $S\subset\R^n$ be a centrally symmetric convex body. Then
\[\sum_{u\in S\cap\Z^n} \rho_\Gamma(u) \ge D^+(\Gamma) \Leb(S/2).\]
\end{theorem}

This theorem brings a new insight to the classical Minkowski theorem. It shows that the object of interest is in fact the set of differences of elements in $\Gamma$ (which, for a lattice, is equal to $\Gamma$). In some sense, this point of view is already present in the original proof of Minkowski and the proof proposed in the sequel critically uses this fact. This is the purpose of Section \ref{sec:MinkThm}.

In Section \ref{sec:wapsets}, we define \emph{weakly almost periodic} sets (see Definition~\ref{wap}). For such sets, the uniform upper density and the frequency of differences are defined as limits (and no longer as upper limits). Roughly speaking, a set $\Gamma$ is weakly almost periodic if given any ball $B$ large enough, the intersection of $\Gamma$ with any translate $t(B)$ of $B$ is a translation of $B\cap \Gamma$ up to a proportion of points $\varep$ arbitrarily small. Such sets include a large class of quasicrystals, in particular model sets (and, of course, lattices).

The remaining part of the paper is dedicated to two applications of our main theorem.
We investigate Diophantine approximation in Section \ref{sec:app_dio}. A corollary is derived to show the existence of a couple of points in a quasicrystal for which the slope of the line defined by those points is arbitrarily close to a fixed, chosen slope. Another application is also considered: for any given irrational number $\alpha$ and any positive number $\varep$, the set $E_\alpha^\varep$ of integers $n$ such that $n\alpha$ is $\varep$-close to 0 is a weakly almost periodic set. Hence, estimates of the mean number of points in $E_\alpha^\varep$ which lie in the ``neighbourhood'' $[x-d,x+d]$ of a point $x\in E_\alpha^\varep$ can be given.

In Section \ref{sec:app_map}, a second application deals with discretizations of linear isometries. In particular, it shows that in most cases, it is impossible not to lose information while performing discrete rotations of numerical images with a naive algorithm.

\section{Definitions}

We begin with a few notations. The symmetric difference of two sets $A$ and $B$ will be denoted by $A \Delta B = (A\setminus B)\cup (B\setminus A)$. The notation $\Delta A$ will be used for the set of differences of $A$, defined as
\[\Delta A = A-A = \{a_1-a_2\mid a_1,a_2 \in A\}.\]
We will use $\card(A)$ for the cardinality of a set $A$, $\lambda$ for the Lebesgue measure on $\R^n$ and $\1$ for the indicator function. The number $\lceil x \rceil$ will denote the smallest integer bigger than $x$. For a set $A\subset \R^n$, we will denote by $\Leb (A)$ the volume of the set $A$. Finally, for any integer $n$, the number $\mu_n$ will refer to the volume of the unit ball of dimension $n$. We will often use the notation $\sum_{x\in A} f(x)$ with $A$ an uncountable set with no further justification; in this paper, every $f$ considered have a countable support.

\begin{definition}
Let $\Gamma$ be a subset of $\R^n$.
\begin{itemize}
\item We say that $\Gamma$ is \emph{relatively dense} if there exists $R_\Gamma>0$ such that each ball with radius at least $R_\Gamma$ contains at least one point of $\Gamma$.
\item We say that $\Gamma$ is \emph{uniformly discrete} if there exists $r_\Gamma>0$ such that each ball with radius at most $r_\Gamma$ contains at most one point of $\Gamma$.
\end{itemize}
The set $\Gamma$ is called a \emph{Delone set} if it is both relatively dense and uniformly discrete.
\end{definition}

\begin{definition}\label{uud}
For a discrete set $\Gamma\subset \R^n$ and $R\ge 1$, the \emph{uniform $R$-density} is:
\[D_R^+(\Gamma) = \sup_{x\in\R^n} \frac{\card\big(B(x,R)\cap \Gamma\big)}{\Leb\big(B(x,R)\big)},\]
and the \emph{uniform upper density} is:
\[D^+(\Gamma) = \underset{R\to +\infty}{\overline\lim} D_R^+(\Gamma).\]
\end{definition}

Remark that if $\Gamma\subset \R^n$ is a Delone set for the parameters $r_\Gamma$ and $R_\Gamma$, then its upper density satisfies:
\[\frac{1}{\mu_n R_{\Gamma}^n} \le D^+(\Gamma) \le \frac{1}{\mu_n r_{\Gamma}^n}.\]

\begin{definition}\label{MeyerSet}
A \emph{Meyer set} $\Gamma$ is a Delone set whose set of differences $\Delta\Gamma$ is also Delone.
\end{definition}

J.C.~Lagarias showed in \cite{MR1400744} that a Delone set is Meyer if and only if $\Delta\Gamma\subset \Gamma+F$ for some finite set $F$. A basic example of Meyer set is a relatively dense subset of a lattice\footnote{Recall that a \emph{lattice} of $\R^n$ is a discrete subgroup of $\R^n$ with finite upper density and which linearly spans $\R^n$.}. More generally, Y.~Meyer showed in \cite{MR0485769} that a Delone set is Meyer\footnote{For the equivalent definition given by the result of J.C.~Lagarias.} if and only if there exists a model set $\Lambda$ (see Definition~\ref{DefModel}) and a finite set $F$ such that $\Gamma\subset \Lambda + F$.
By definition, the set of differences of a Meyer set has finite upper density. The following definition quantifies the density of differences in the set $\Gamma$ and compares it to the density of $\Gamma$.

\begin{definition}\label{DefDiff}
For every $v\in\Z^n$, we set\index{$\rho_\Gamma$}
\[\rho_\Gamma(v) = \frac{D^+\{x\in\Gamma\mid x+v\in\Gamma\}}{D^+(\Gamma)} = \frac{D^+\big(\Gamma\cap(\Gamma-v)\big)}{D^+(\Gamma)} \in [0,1]\]
the \emph{frequency} of the difference $v$ in the Delone set $\Gamma$.
\end{definition}
Remark that when $\Gamma$ is a lattice, the set $\Gamma\cap(\Gamma-v)$ is either equal to $\Gamma$ (when $v\in \Gamma$), either empty (when $v \notin \Gamma$). Hence, $\rho_\Gamma(v)= \1_{v\in \Gamma}$ and for any subset $A$ of $\R^n$, $\sum_{v\in A} \rho_{\Gamma}(v)$ counts the number of elements of $\Gamma$ falling in $A$.

\begin{definition}\label{mine}
We say that the function $f$ admits a \emph{mean} $\mathcal M(f)$ if for every $\varep>0$, there exists $R_0>0$ such that for every $R\ge R_0$ and every $x\in\R^n$, we have
\[\left|\mathcal M(f) - \frac{1}{\Leb\big(B(x,R)\big)}\sum_{v\in B(x,R)} f(v) \right| <\varep.\]
\end{definition}

\section{A Minkowski theorem for Meyer sets}
\label{sec:MinkThm}

We now state a Minkowski theorem for the map $\rho_\Gamma$. To begin with, we recall the classical Minkowski theorem which is only valid for lattices (see for example IX.3 of \cite{MR2724440} or the whole books \cite{MR893813,MR1020761,MR1434478}).

\begin{theoreme}[Minkowski]\label{Minkowski}
Let $\Lambda$ be a lattice of $\R^n$, $k\in\N$ and $S\subset\R^n$ be a centrally symmetric convex body. If $\Leb(S/2) > k \operatorname{covol}(\Lambda)$, then $S$ contains at least $2k$ distinct points of $\Lambda\setminus\{0\}$.
\end{theoreme}

In particular, if $\Leb(S/2) > \operatorname{covol}(\Lambda)$, then $S$ contains at least one point of $\Lambda\setminus\{0\}$. This theorem is optimal in the following sense: for every lattice $\Lambda$, there exists a centrally symmetric convex body $S$ such that $\Leb(S/2) = k \operatorname{covol}(\Lambda)$ and that $S$ contains less than $2k$ distinct points of $\Lambda\setminus\{0\}$.

\begin{proof}[Proof of Theorem \ref{Minkowski}]
We consider the integer valued function
\[\varphi = \sum_{\lambda\in\Lambda} \1_{\lambda + S/2}.\]
The hypothesis about the covolume of $\Lambda$ and the volume of $S/2$ imply that the mean of the periodic function $\varphi$ satisfies $\mathcal M(\varphi)>k$. In particular, as $\varphi$ has integer values, there exists $x_0\in \R^n$ such that $\varphi(x_0)\ge k+1$ (note that this argument is similar to pigeonhole principle). So there exists $\lambda_0,\cdots,\lambda_k\in \Lambda$, with the $\lambda_i$ sorted in lexicographical order (for a chosen basis), such that the $x_0-\lambda_i$ all belong to $S/2$. As $S/2$ is centrally symmetric, as $\lambda_i-x_0$ belongs to $S/2$ and as $S/2$ is convex, $\big((x_0-\lambda_0) + (\lambda_i-x_0)\big)/2 = (\lambda_i - \lambda_0)/2$ also belongs to $S/2$. Then, $\lambda_i-\lambda_0\in (\Lambda\setminus\{0\}) \cap S$ for every $i\in \{ 1,\cdots,k\}$. By hypothesis, these $k$ vectors are all different. To obtain $2k$ different points of $S\cap \Lambda\setminus\{0\}$ (instead of $k$ different points), it suffices to consider also the points $\lambda_0-\lambda_i$; this collection is disjoint from the collection of $\lambda_i-\lambda_0$ because the $\lambda_i$ are sorted in lexicographical order. This proves the theorem.
\end{proof}

Minkowski theorem can be seen as a result about the function $\rho_\Gamma$. Recall that for a lattice $\Lambda$, $\sum_{u\in S} \rho_{\Lambda}(u)$ equals exactly the number of elements of $S\cap \Gamma$. Then, for a centrally symmetric convex body $S\subset\R^n$, 
\[
\sum_{u\in S} \rho_\Lambda(u) \ge 2\lceil D(\Lambda) \Leb(S/2)\rceil-1.
\]
Simply remark that the optimal $k$ in Theorem~\ref{Minkowski} is given by $k = \lceil D(\Lambda) \Leb(S/2)\rceil-1$. The following result is the main theorem of the paper.

\begin{theoreme}\label{MinkAlm}
Let $\Gamma$ be a Meyer subset of $\R^n$, and $S\subset\R^n$ be a centrally symmetric convex body. Then
\[\sum_{u\in S} \rho_\Gamma(u) \ge D^+(\Gamma) \Leb(S/2).\]
\end{theoreme}




\begin{rem}
One can note that this theorem does not involve the factor 2 present in the classical Minkowski theorem which results of the fact that to any point of a lattice falling in the centrally symmetric set $S$ corresponds its opposite, which also lies in $S$. The authors do not know if this factor 2 should or should not be present in Theorem~\ref{MinkAlm} and this fact still remains to be investigated.
\end{rem}

\begin{proof}[Proof of Theorem \ref{MinkAlm}]
The strategy of proof of this theorem is similar to that of the classical Minkowski theorem: we consider the set $\Gamma + S/2$ and define a suitable auxiliary function which depends on this set.
The argument is based on a double counting for the quantity
\begin{equation}\label{derEq}
\rho_a^R = \frac{1}{\Leb(B_R)}\sum_{v\in B_R\cap\Gamma} \1_{v\in(S/2+a)} \sum_{u\in S\cap\DG}\frac{\1_{v\in\Gamma} \1_{u+v\in\Gamma}}{D^+(\Gamma)}.
\end{equation}
When $R$ is large, $\rho_a^R$ can be interpreted as the approximate frequency of the differences falling in $S$ with the restriction that one of the point (in the difference) is in $S/2+a$. It can also be interpreted as the local approximate frequency in a neighbourhood of $a$ (the neighbourhood $S/2+a$). The convenience of this restriction is expressed in Equation \eqref{eq:minkov_argu}. A way to get a global expression, from this local definition of the frequency, is to sum over $a\in \R^n$.
\begin{align*}
\int_{\R^n} \rho_a^R d\lambda (a) & = \frac{1}{\Leb(B_R)}\sum_{v\in B_R\cap\Gamma} \sum_{u\in S\cap\DG}\frac{\1_{v\in\Gamma} \1_{u+v\in\Gamma}}{D^+(\Gamma)}\int_{\R^n}\1_{a\in(S/2+v)} d\lambda(a)\\
    & = \frac{1}{\Leb(B_R)}\sum_{v\in B_R\cap\Gamma} \sum_{u\in S\cap\DG}\frac{\1_{v\in\Gamma} \1_{u+v\in\Gamma}}{D^+(\Gamma)}\Leb(S/2)\\
    & = \Leb(S/2)\sum_{u\in S\cap\DG} \frac{1}{\Leb(B_R)}\sum_{v\in B_R\cap\Gamma} \frac{\1_{v\in\Gamma} \1_{u+v\in\Gamma}}{D^+(\Gamma)}.
\end{align*}
Thus, by the definition of $\rho_\Gamma$, we get 
\begin{equation}\label{res1}
\underset{R\to +\infty}{\overline\lim} \int_{\R^n} \rho_a^R d\lambda (a)\, \le\, \Leb(S/2)\sum_{u\in S\cap\DG} \rho_\Gamma(u).
\end{equation}

In sight of the last inequality, it remains to show a lower bound on the left hand side.
First of all, we remark that as $S$ is a centrally symmetric convex body, $x,y\in S/2$ implies that $x+y\in S$, thus
\begin{equation}
\label{eq:minkov_argu}
\1_{v\in S/2+a} \1_{u\in S}  \ge \1_{v\in S/2+a} \1_{u+v\in S/2+a}.
\end{equation}
Hence, multiplying both sides by $\1_{v\in\Gamma} \1_{u+v\in\Gamma} $, we get 
\[\1_{v\in(S/2+a)\cap\Gamma} \1_{u\in S} \1_{u+v\in\Gamma} \ge \1_{v\in(S/2+a)\cap\Gamma} \1_{u+v\in(S/2+a)\cap\Gamma}.\]
We now sum this inequality over $u\in\DG$ to get
\[
\sum_{u\in S\cap\DG} \1_{v\in(S/2+a)\cap\Gamma}\1_{u+v\in\Gamma} \ge \1_{v\in(S/2+a)\cap\Gamma} \sum_{u\in\DG}\1_{u+v\in(S/2+a)\cap\Gamma}.
\]
Remarking that for every $v\in\Gamma$, every $v'\in\Gamma$ can be written as $v'=u+v$ with $u\in\DG$, we deduce that
\begin{align*}
\sum_{u\in S\cap\DG} \1_{v\in(S/2+a)\cap\Gamma}\1_{u+v\in\Gamma} & \ge \1_{v\in (S/2+a)\cap\Gamma} \sum_{v'\in\Gamma}\1_{v'\in(S/2+a)\cap\Gamma}\\
     & \ge \1_{v\in (S/2+a)\cap\Gamma} \card\big((S/2+a)\cap\Gamma\big),
\end{align*}
and finally,
\[\rho_a^R \ge \frac{1}{D^+(\Gamma)}\frac{1}{\Leb(B_R)}\sum_{v\in B_R\cap\Gamma}\1_{v\in(S/2+a)} \card\big((S/2+a)\cap\Gamma\big).\]
We denote by $B_R^S$ the $S$-interior of $B_R$ and by $\BS$ the $S$-expansion of $B_R$,
\begin{align*}
	B_R^S &= \big(B_R^\complement + S\big)^\complement = \{x\in B_R\mid\forall s\in S, x+s\in B_R\}\\
	\BS &= B_R + S = \{x +s \mid x\in B_R,  s\in S\}
\end{align*}
In particular, $a\in B_R^S$ implies that $S/2+a\subset B_R$ and $a\in B_R$ implies that $S/2+a \in \BS$. Then
\begin{align*}
\int_{\R^n} \rho_a^R d\lambda (a) & \ge \frac{1}{D^+(\Gamma)} \frac{1}{\Leb(B_R)} \int_{\R^n} \left(\sum_{v\in B_R\cap\Gamma}\1_{v\in(S/2+a)} \card\big((S/2+a)\cap\Gamma\big) \right)d\lambda(a)\\
    & \ge \frac{1}{D^+(\Gamma)}\frac{1}{\Leb(B_R)}\int_{B_R^S}\left(\sum_{v\in B_R\cap\Gamma} \1_{v\in(S/2+a)} \card\big((S/2+a)\cap\Gamma\big)\right)d\lambda(a)\\
		& \ge \frac{1}{D^+(\Gamma)}\frac{1}{\Leb(B_R)}\int_{B_R^S} \card\big((S/2+a)\cap\Gamma\big)^2d\lambda(a).
\end{align*}
Using the convexity of $x\mapsto x^2$, we deduce that
\begin{equation}\label{res2}
\underset{R\to +\infty}{\overline\lim} \int_{\R^n} \rho_a^R d\lambda(a) \ge \underset{R\to +\infty}{\overline\lim}\ \frac{\Leb(B_R^S)}{D^+(\Gamma)\Leb(B_R)}\left(\frac{1}{\Leb(B_R^S)}\int_{B_R^S} \card\big((S/2+a)\cap \Gamma\big)d\lambda(a)\right)^2.
\end{equation}
We then use the fact that the family $\{B_R\}_{R>0}$ is van Hove when $R$ goes to infinity (see for example \cite[Equation 4]{MR1884143}), that is
\begin{align*}
	\lim_{R\to +\infty} \frac{\Leb(B_R)-\Leb(B_R^S)}{\Leb(B_R)} = 0 \quad \text{and} \quad
	\lim_{R\to +\infty} \frac{\Leb(B_R)-\Leb(\BS)}{\Leb(B_R)} = 0.
\end{align*}
It remains to compute the remaining term in Equation~\eqref{res2},
\[
\frac{1}{\Leb(B_R^S)}\int_{B_R^S} \card\big((S/2+a)\cap\Gamma\big)d\lambda(a).
\]
The quantity $\card\big((S/2+a)\cap\Gamma\big)$ is bounded by some constant $M$ (as $S$ can be included in some ball of large radius), independently from $a$ and is equal to 
\[
\sum_{v\in B_R\cap\Gamma} \1_{v\in (S/2+a)\cap\Gamma} \quad \text{ for all } a \in B_R^S.
\]
Hence,
\[
 \frac{1}{\Leb(B_R^S)}\left|\int_{B_R^S} \card\big((S/2+a)\cap\Gamma\big)d\lambda(a)-\int_{\BS}\sum_{v\in B_R\cap\Gamma} \1_{v\in S/2+a} d\lambda(a)\right| \le M \frac{\Leb(\BS \setminus B_R^S)}{\Leb(B_R^S)};
\]
thus the two integrals have the same limit superior when $R$ tends to $+\infty$. Besides,
\begin{align*}
\frac{1}{\Leb(B_R^S)}\int_{\BS}\sum_{v\in B_R\cap\Gamma} \1_{v\in S/2+a} d\lambda(a) & = \frac{1}{\Leb(B_R^S)}\sum_{v\in B_R\cap\Gamma} \int_{\BS} \1_{a\in S/2+v}d\lambda(a)\\
      & = \frac{1}{\Leb(B_R^S)}\sum_{v\in B_R\cap\Gamma} \Leb(S/2)\\
      & = \frac{\Leb(B_R)}{\Leb(B_R^S)}\frac{\card( B_R\cap\Gamma)}{\Leb(B_R)} \Leb(S/2).
\end{align*}

Applied to Equation \eqref{res2}, this gives 
\[\underset{R\to +\infty}{\overline\lim} \int_{\R^n} \rho_a^R d\lambda(a) \ge \Leb(S/2)^2 D^+(\Gamma).\]

To finish the proof, we combine the last inequality with the first estimate of Equation \eqref{res1} and get
\[\sum_{u\in S\cap\DG} \rho_\Gamma(u) \ge \Leb(S/2) D^+(\Gamma).\]
\end{proof}

\section{Weakly almost periodic sets}
\label{sec:wapsets}

In this section, we describe a family of Meyer sets called \emph{weakly almost periodic sets}. For these sets, the superior limits appearing in the definitions of the upper density and the frequency of differences are actually limits. Roughly speaking, a weakly almost periodic set $\Gamma$ is a set for which two large patches are almost identical, up to a set of upper density smaller than $\varep$. More precisely, we have the following definition.

\begin{definition}\label{wap}
We say that a Delone set $\Gamma$ is \emph{weakly almost periodic} if for every $\varep>0$, there exists $R>0$ such that for every $x,y\in\R^n$, there exists $v\in\R^n$ such that
\begin{equation}\label{EqWeakAlmPer}
\frac{\card\Big( \Big(B(x,R)\cap\Gamma\Big) \Delta \Big(\big(B(y,R)\cap\Gamma\big)-v\Big) \Big)}{\Leb(B_R)} \le \varep.
\end{equation}
\end{definition}

Note that the vector $v$ is different from $y-x$ \emph{a priori}. The Delone set assumption is not restrictive in the later definition. Indeed, any $\Gamma \subset\Z^n$ with positive upper density ($D^+(\Gamma) >0$) and satisfying Equation~\eqref{EqWeakAlmPer} is a Delone set.
Of course, every lattice, or every finite union of translates of a given lattice, is weakly almost periodic.

A weakly almost periodic set possesses a uniform density, as stated by the following proposition of \cite{Gui15d}.

\begin{prop}\label{limitexist}
Let $\Gamma$ be a weakly almost periodic set. Then there exists a number $D(\Gamma)$, called the \emph{uniform density} of $\Gamma$, satisfying: for every $\varep>0$, there exists $R_\varep>0$ such that for every $R>R_\varep$ and every $x\in \R^n$,
\[\left| \frac{\card\big(B(x,R)\cap \Gamma\big)}{\Leb\big(B(x,R)\big)} - D(\Gamma) \right| < \varep.\]
In particular, $D(\Gamma) = D^+(\Gamma)$, and for every $x\in\R^n$, we have
\[D(\Gamma) = \lim_{R\to +\infty}\frac{\card\big(B(x,R)\cap \Gamma\big)}{\Leb\big(B(x,R)\big)}.\]
\end{prop}

As noted in \cite{Gui15d}, it seems that the notion of weakly almost periodicity is the weakest that allows this uniform convergence of density.
\bigskip

An important class of examples of weakly almost periodic sets is given by \emph{model sets} (sometimes also called ``cut-and-project'' sets). These sets have numerous applications to theory of quasicrystals, harmonic analysis, number theory, discrete dynamics etc. (see for instance \cite{MR0485769} or \cite{Moody25}).

\begin{definition}\label{DefModel}
Let $\Lambda$ be a lattice of $\R^{m+n}$, $p_1$ and $p_2$ the projections of $\R^{m+n}$ on respectively $\R^m\times \{0\}_{\R^n}$ and $\{0\}_{\R^m} \times \R^n$, and $W$ a Riemann integrable subset of $\R^m$. The \emph{model set} modelled on the lattice $\Lambda$ and the \emph{window} $W$ is (see Figure~\ref{FigModel})
\[\Gamma = \big\{ p_2(\lambda)\mid \lambda\in\Lambda,\, p_1(\lambda)\in W \big\}.\]
\end{definition}

\begin{figure}[t]
\begin{center}
\begin{tikzpicture}[scale=1]
\fill[color=blue!10!white] (-.6,-2) rectangle (.9,2);
\draw[color=blue!80!black] (-.6,-2) -- (-.6,2);
\draw[color=blue!80!black] (.9,-2) -- (.9,2);
\draw[color=blue!80!black, very thick] (-.6,0) -- (.9,0);
\draw[color=blue!80!black] (.25,0) node[below] {$W$};
\draw (-3,0) -- (3,0);
\draw (0,-2) -- (0,2);
\clip (-3,-2) rectangle (3,2);

\draw (.866,.364) -- (0,.364);
\draw (-.129,.987) -- (0,.987);
\draw (.737,1.351) -- (0,1.351);
\draw (-.258,1.974) -- (0,1.974);
\draw (.129,-.987) -- (0,-.987);
\draw (.258,-1.974) -- (0,-1.974);

\draw (0,0) node {$\times$};
\draw (0,.364) node {$\times$};
\draw (0,.987) node {$\times$};
\draw (0,1.351) node {$\times$};
\draw (0,1.974) node {$\times$};
\draw (0,-.987) node {$\times$};
\draw (0,-1.974) node {$\times$};

\foreach\i in {-3,...,3}{
\foreach\j in {-3,...,3}{
\draw[color=green!40!black] (.866*\i-.129*\j,.364*\i+.987*\j) node {$\bullet$};
}}
\draw[color=green!40!black] (1.2,-.8) node {$\Lambda$};
\end{tikzpicture}
\caption[Model set]{Construction of a model set.}\label{FigModel}
\end{center}
\end{figure}
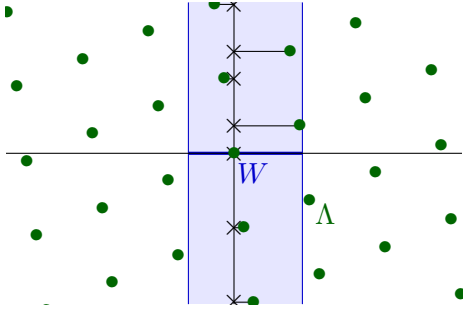

In \cite{Gui15d} it is proved that these sets are weakly almost periodic sets. Moreover, if the projection $p_2$ is injective when restricted to $\Lambda$, and the set $p_2(\Lambda)$ is dense, then the density of the obtained model set is equal to $\operatorname{Vol}(W)\operatorname{Covol}(\Lambda)$ (see for example Proposition 4.4 of \cite{MR2876415}).
\bigskip

From now, we suppose that the weakly almost periodic sets we consider are Meyer. The following lemma states that the occurrences of a given difference in a weakly almost periodic set form a weakly almost periodic set.

\begin{lemme}\label{Moule}
Let $v\in\R^n$ and $\Gamma$ be a Meyer, weakly almost periodic set. Then the set
\[\{x\in\Gamma\mid x+v\in\Gamma\} = \Gamma\cap(\Gamma-v)\]
is weakly almost periodic.
\end{lemme}

\begin{proof}[Proof of Lemma~\ref{Moule}]
Since $\Gamma$ is Meyer, its set of differences is Delone. 

Let $\varep>0$ and $v\in\R^n$. As $\Gamma$ is a weakly almost periodic set, for every $\varep>0$, there exists $R>0$ such that for every $x,y\in\R^n$, there exists $w\in\R^n$ such that
\begin{equation}\label{eqmoule1}
\frac{\card\Big( \Big(B(x,R)\cap\Gamma\Big) \Delta \Big(\big(B(y,R)\cap\Gamma\big)-w\Big) \Big)}{\Leb(B_R)} \le \varep.
\end{equation}
On the other hand, considering a smaller $\varep$ if necessary, one can choose $R$ arbitrarily large compared to $\|v\|$. In this case, we have, for every $z\in\R^n$ (and in particular for $x$ and $y$),
\begin{equation}\label{eqmoule2}
\frac{\card\Big( \Big(B(z,R)\cap\Gamma\Big) \Delta \Big(\big(B(z+v,R)\cap\Gamma\big) \Big) \Big)}{\Leb(B_R)} \le \frac{\card\Big( B(z,R) \Delta B(z+v,R)\Big)}{\Leb(B_R)} \le \varep.
\end{equation}

Let $x,y\in\R^n$. We can now estimate the quantity
\[A \doteq \frac{\card\Big( \Big(B(x,R)\cap\Gamma\cap(\Gamma-v)\Big) \Delta \Big(\big(B(y,R)\cap\Gamma\cap(\Gamma-v)\big)-w\Big) \Big)}{\Leb(B_R)}:\]
\begin{align*}
A \le & \frac{\card\Big( \Big(B(x,R)\cap\Gamma\Big) \Delta \Big(\big(B(y,R)\cap\Gamma\big)-w\Big) \Big)}{\Leb(B_R)}\\
      & + \frac{\card\Big( \Big(B(x,R)\cap(\Gamma-v)\Big) \Delta \Big(\big(B(y,R)\cap(\Gamma-v)\big)-w\Big) \Big)}{\Leb(B_R)}.
\end{align*}
The first term is smaller than $\varep$ by Equation~\eqref{eqmoule1}; and the second term (denoted by $A_2$) is smaller than (by a translation of vector $v$)
\[A_2 \le \frac{\card\Big( \Big(B(x+v,R)\cap\Gamma\Big) \Delta \Big(\big(B(y+v,R)\cap\Gamma\big)-w\Big) \Big)}{\Leb(B_R)},\]
which by Equations~\eqref{eqmoule1} and \eqref{eqmoule2} leads to
\[A_2 \le 3\varep.\]
Finally, $A\le 4\varep$.
\end{proof}

When $\Gamma$ is weakly almost periodic, we deduce, from Lemma~\ref{Moule} together with Proposition \ref{limitexist}, that the upper limits appearing in the uniform upper density $D^+(\Gamma)$ and the frequencies of differences $\rho_\Gamma(v)$ are, in fact, limits. Moreover, $\rho_\Gamma$ possesses a mean (see Definition~\ref{mine}) that can be computed easily.

\begin{prop}\label{IntRho}
If $\Gamma$ is Meyer and weakly almost periodic, then 
\[\mathcal M (\rho_\Gamma) = D(\Gamma).\]
\end{prop}

\begin{proof}[Proof of Proposition \ref{IntRho}]
This proof lies primarily in an inversion of limits.

Let $\varep>0$. As $\Gamma$ is weakly almost periodic, by Proposition~\ref{limitexist}, there exists $R_0>0$ such that for every $R\ge R_0$ and every $x\in\R^n$, we have
\begin{equation}\label{eqDens}
\left|D(\Gamma) - \frac{\Gamma \cap B(x,R)}{\Leb(B_R)}\right|\le \varep.
\end{equation}

So, we choose $R\ge R_0$, $x\in\Z^n$ and compute
\begin{align*}
\frac{1}{\Leb(B_R)} & \sum_{v\in B(x,R)} \rho_\Gamma(v) = \frac{1}{\Leb(B_R)}\sum_{v\in B(x,R)} \frac{D\big((\Gamma-v)\cap \Gamma\big)}{D(\Gamma)}\\
       = & \frac{1}{\Leb(B_R)}\sum_{v\in B(x,R)} \lim_{R'\to +\infty}\frac{1}{\Leb(B_{R'})}\sum_{y\in B_{R'}} \frac{\1_{y\in\Gamma-v} \1_{y\in\Gamma}}{D(\Gamma)}\\
       = & \frac{1}{D(\Gamma)}\lim_{R'\to +\infty}\frac{1}{\Leb(B_{R'})} \sum_{y\in B_{R'}} \1_{y\in\Gamma}\frac{1}{\Leb(B_R)}\sum_{v\in B(x,R)} \1_{y\in\Gamma-v}\\
       = & \frac{1}{D(\Gamma)}\underbrace{\lim_{R'\to +\infty}\frac{1}{\Leb(B_{R'})} \sum_{y\in B_{R'}} \1_{y\in\Gamma}}_{\text{first term}}\underbrace{\frac{1}{\Leb(B_R)}\sum_{v'\in B(y+x,R)} \1_{v'\in\Gamma}}_{\text{second term}}.
\end{align*}
By Equation \eqref{eqDens}, the second term is $\varep$-close to $D(\Gamma)$. Considered independently, the first term is equal to $D(\Gamma)$ (still by Equation \eqref{eqDens}). Thus, we have
\[\left|\frac{1}{\Leb\big(B(x,R)\big)} \sum_{v\in B(x,R)} \rho_\Gamma(v) - D(\Gamma)\right|\le \varep,\]
which conclude the proof.
\end{proof}

Theorem \ref{MinkAlm} may now be reformulated in the context of Meyer weakly almost periodic sets. 

\begin{coro}\label{CoroMinkAlm}
If $\Gamma \subset \Z^n$ is a weakly almost periodic set, then
\[
\sum_{u\in S} \rho_\Gamma(u) \ge D(\Gamma) \card(S/2\cap\Z^n).
\]
\end{coro}

The idea of the proof of this corollary is identical to that of Theorem \ref{MinkAlm}, but instead of integrating $\rho_a^R$ (see Equation~\eqref{derEq}) over $\R^n$, one sums $\rho_a^R$ over $\Z^n$. Some technicalities in the proof require $\Gamma$ to be a weakly almost periodic subset of $\Z^n$.

The case of equality in this corollary is attained even in the non trivial case where $\card(S/2\cap\Z^n)>1$, as shown by the following example.
\begin{ex}\label{OptMink}
If $k$ is an odd number, if $\Gamma$ is the lattice $k\Z\times \Z$, and if $S$ is a centrally symmetric convex set such that (see Figure~\ref{FigExMink})
\[S\cap \Gamma = \big\{(i,0)\mid i\in \{-(k-1),\cdots, k-1\}\big\} \cup \big\{\pm(i,1)\mid i\in\{ 1,\cdots, k-1\}\big\} ,\]
then $\sum_{u\in S} \rho(u) = 1$, $D(\Gamma) = 1/k$ and $\card(S/2 \cap \Z^n) = k$.
\end{ex}

\begin{figure}[t]
\begin{center}
\begin{tikzpicture}[scale=.75]
\draw (0,0) node[right]{$0$};
\foreach\i in {-1,...,1}{
\foreach\j in {-2,...,2}{
\draw[color=red!70!black] (3*\i,\j) node {$\bullet$};
}}
\foreach\i in {-3,...,3}{
\foreach\j in {-2,...,2}{
\draw[color=black] (\i,\j) node {$\cdot$};
}}
\draw[color=green!70!black,thick] (-2.2,.1) -- (-2.2,-1.2) -- (-1,-1.2) -- (2.2,-.1) -- (2.2,1.2) -- (1,1.2) -- cycle;
\draw[color=green!40!black] (1,1.1) node[above right] {$S$};
\draw[color=blue!70!black] (-1.1,.05) -- (-1.1,-.6) -- (-.5,-.6) -- (1.1,-.05) -- (1.1,.6) -- (.5,.6) -- cycle;
\draw[color=blue!40!black] (1,.5) node[right] {\small$S/2$};
\draw[color=red!60!black] (3,.3) node[right] {$\Gamma$};
\end{tikzpicture}
\caption{Example \ref{OptMink} of equality case in Corollary~\ref{CoroMinkAlm} for $k=3$.}\label{FigExMink}
\end{center}
\end{figure}
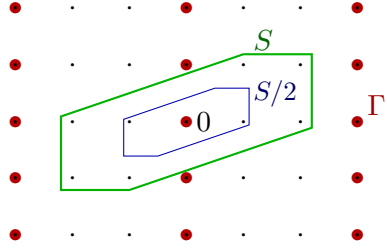

\section{Applications}
\label{sec:appli}

\subsection{Application to Diophantine approximation}
\label{sec:app_dio}

\subsubsection{A Dirichlet theorem for quasicrystals}

In this section, we develop a generalization of Dirichlet theorem for approximations of irrational numbers. We give this theorem for completeness.

\begin{theoreme}[Dirichlet]
Let $\overline\alpha = (\alpha_1,\dots,\alpha_n)$ be such that at least one of the $\alpha_i$ is irrational. Then there are infinitely many tuples of integers $(x_1,\dots,x_n,y)$ such that the highest common factor of $x_1,\dots,x_n, y$ is 1 and that
 	\[
 	\left|\frac{x_i}{y}-\alpha_i\right| \le y^{-1-1/n} \text{ for } i=1,\dots,n.
 	\]
\end{theoreme}

One may be interested by approximations of real numbers by tuples in sets different from $\Z^{n+1}$, for instance quasicrystals. The following result is an easy consequence of Theorem \ref{MinkAlm} which is convenient for the study of Diophantine approximations in weakly almost periodic sets.

\begin{coro}
\label{co:dirichlet}
	Let $L_1,\dots,L_n$ be $n$ linear forms on $\R^n$ such that $\det (L_1,\dots,L_n) \neq 0$. Let $A_1,\dots,A_n$ be positive real numbers and let $\Gamma$ be a weakly almost periodic set. Then
	\[
	\sum_{\substack{x\in \Z^n \\ \forall i\ |L_i(x)|\le A_i}}\!\! \rho_{\Gamma}(x) \ge D(\Gamma) A_1\dots A_n |\det (L_1,\dots,L_n)|^{-1}.
	\]
\end{coro}

\begin{proof}
	Define $S = \{x\in \R^n : |L_i(x)| \le A_i \}$ .
\end{proof}

Let $Q>1$ be a parameter that will be specified later and define for any $i \in \{1,\dots,n\}$ the linear form $L_i:\R^{n+1} \rightarrow \R$ by
\[
L_i(x_1,\dots,x_n,y)=x_i-\alpha_i y
\]
and $L_{n+1}(x_1,\dots,x_n,y) = y$. We have immediately that $\det (L_1,\dots,L_n)=1$. We apply Corollary \ref{co:dirichlet} with $A_1=\dots=A_n=Q^{-1/n}$ and $A_{n+1}=2 Q/D(\Gamma)$. If all the inequalities $|L_i(x_1,\dots,x_n,y)| \le A_i$ are verified, then for each $i\le n$,
\[
\left|\alpha_i- \frac{x_i}{y}\right| \le \frac{2^{1/n}}{D(\Gamma)^{1/n}}|y|^{-1-1/n}.
\]
We obtain
\[
\sum_{\substack{(x_1,\dots,x_n,y)\in \Z^{n+1}\setminus \{0\} \\ \forall i,\ |L_i(x)|\le A_i}}\!\! \rho_{\Gamma}(x) \ge 1.
\]
In particular, there exists a point $u$ in $\Z^{n+1}$ which is a difference of two different points $v=(x_1^v,\dots,x_n^v,y^v)$ and $w=(x_1^w,\dots,x_n^w,y^w)$ in the weakly almost periodic set $\Gamma$ and such that $\forall i$, $|L_i(v-w)|\le A_i$. Thus, for any $i\le n$, the slope of the line defined by the two points $(x_i^v,y^v)$ and $(x_i^w,y^w)$ approximates the slope given by $\overline\alpha$. More precisely,
\[
\left|\alpha_i- \frac{x_i^v-x_i^w}{y^v-y^w}\right| \le \frac{2^{1/n}}{D(\Gamma)^{1/n}}|y^v-y^w|^{-1-1/n} \le \frac{D(\Gamma)2^{1/n}}{(4Q)^{1+1/n}}.
\]
Remark that the approximation quality highly depends on the density of the considered set $\Gamma$. Thus, we will find at least one direction in $\Gamma$ close to $\overline\alpha$ (close with a factor comparable to $Q^{-1-1/n}$) in one ball of size comparable with $Q$. This can be seen as a non-asymptotic counterpart of the strong results of \cite{MR3384490}.

\begin{figure}[t]
\begin{center}
\includegraphics[width=.6\linewidth]{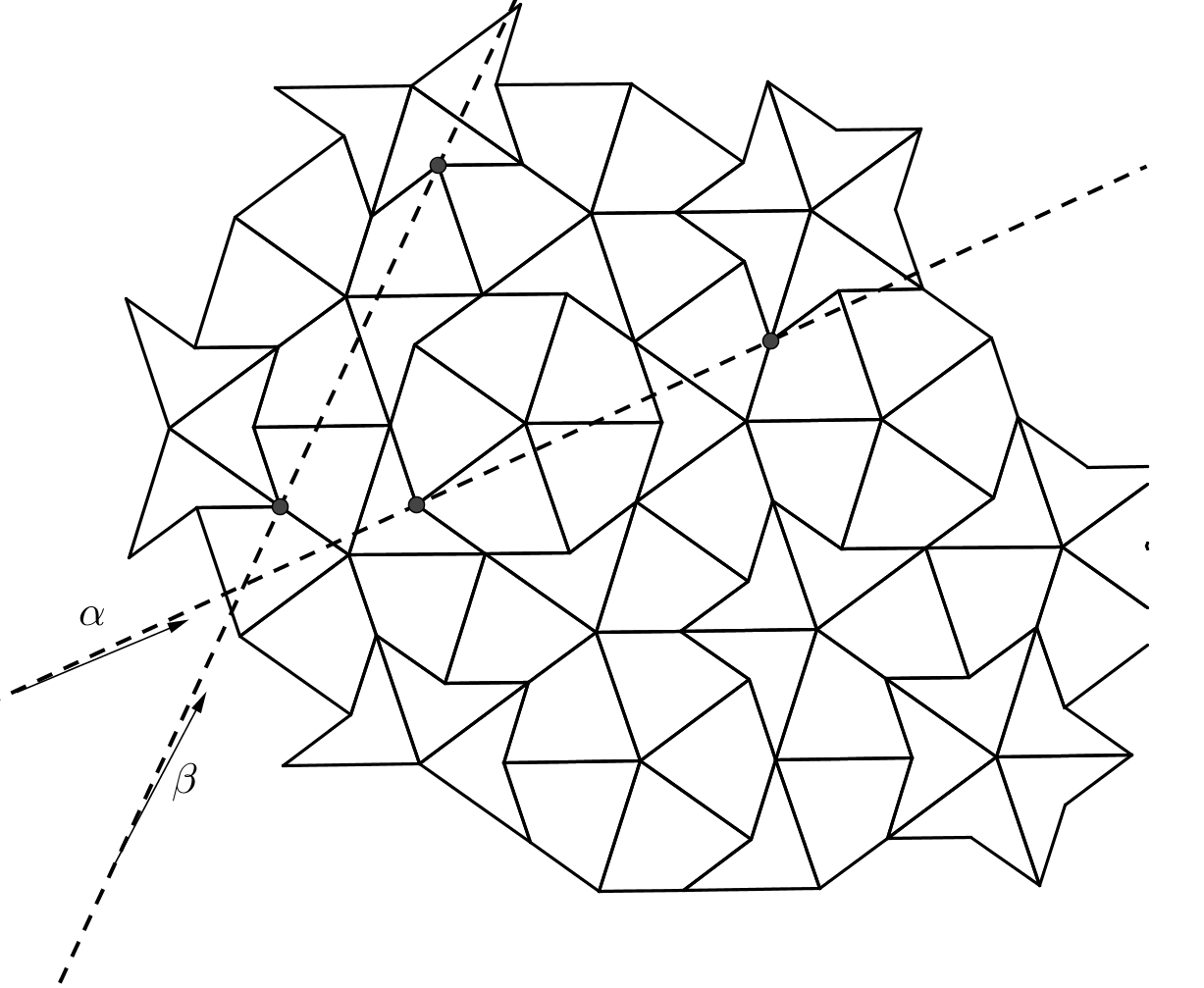}
\caption{For a fixed chosen direction $\alpha$, one can find two points in a Penrose tilling defining a line whose slope is close to $\alpha$. Two different chosen directions are shown. Penrose tilings are model sets, thus are weakly almost periodic (see \cite{MR609465}).}
\end{center}	
\end{figure}

\subsubsection{Frequency of differences and approximations}

Theorem~\ref{MinkAlm} gives informations about the simultaneous approximations of a set of numbers for an arbitrary norm: given a norm $N$ on $\R^n$ and a $n$-tuple of $\Q$-linearly independent numbers $\overline\alpha = (\alpha_1,\cdots,\alpha_n)$, we look at the set
\begin{align*}
E_{\overline\alpha}^\varep & = \big\{y\in\Z\mid \exists x\in\Z^n : N(y\overline\alpha - x)<\varep\big\}\\
    & =  \Big\{y\in\Z\mid \exists x\in\Z^n : N\big(\overline\alpha - \frac{x}{y}\big)<\frac{\varep}{y}\Big\}.
\end{align*}

This set is a model set modelled on the lattice spanned by the matrix
\[\begin{pmatrix}
-1  &  &  &  \alpha_1\\
& \ddots & & \vdots\\
& & -1    & \alpha_n\\
& & & 1
\end{pmatrix}\]
and on the window $W=\{x\in\R^n\mid N(x)<\varep\}$. It is a weakly almost periodic set with density equal to $\operatorname{Vol}(W)$ (as long as $W$ does not intersect any integer translate of itself). Then Theorem~\ref{MinkAlm} asserts that for every $d>0$,
\[\sum_{\substack{u\in\Z \\ |u|\le d}} \rho_{E_{\overline\alpha}}^\varep(u) \ge d\operatorname{Vol}(W).\]
In other words, given $v\in E_{\overline\alpha}^\varep$, the average number of points $v'\in E_{\overline\alpha}^\varep$ such that $|v-v'|\le d$ is bigger than $d\operatorname{Vol}(W)$.

\subsection{Application to the dynamics of the discretizations of linear maps}
\label{sec:app_map}

Here, we recall a theorem of \cite{Gui15b} and sketch its proof, which crucially uses Minkowski theorem for weakly 	almost periodic sets.

We take a Euclidean projection\footnote{That is, $\pi(x)$ is (one of the) point(s) of $\Z^n$ the closest from $x$ for the Euclidean norm.} $\pi$ of $\R^n$ onto $\Z^n$; given $A\in GL_n(\R)$, the \emph{discretization} of $A$ is the map $\widehat A = \pi\circ A : \Z^n\to\Z^n$. This is maybe the simplest way to define a discrete analogue of a linear map. We want to study the action of such discretizations on the set $\Z^n$; in particular if these maps are far from being injective, then when applied to numerical images, discretizations will induce a loss of quality in the resulting images.

Thus, we study the \emph{rate of injectivity} of discretizations of linear maps: given a sequence $(A_k)_{k\in\N}$ of linear maps, the \emph{rate of injectivity in time $k$} of this sequence is the quantity
\[\tau^k(A_1,\cdots,A_k) = \lim_{R\to +\infty} \frac{\card \big((\widehat{A_k}\circ\cdots\circ\widehat{A_1}) (B_R\cap\Z^n)\big)}{\card (B_R\cap\Z^n)}\in]0,1].\]
To prove that the limit of this definition is well defined, we show that 
\[\limsup_{R\to +\infty} \frac{\card \big((\widehat{A_k}\circ\cdots\circ\widehat{A_1}) (B_R\cap\Z^n)\big)}{\card (B_R\cap\Z^n)} = |\det(A_1\cdots A_k)| D^+\left( (\widehat{A_k}\circ\cdots\circ\widehat{A_1}) (\Z^n)\right)\]
and use the fact that the set $(\widehat{A_k}\circ\cdots\circ\widehat{A_1}) (\Z^n)$ is weakly almost periodic. In particular, when all the matrices are of determinant $\pm 1$, we have
\[\tau^k(A_1,\cdots,A_k) = D^+\left( (\widehat{A_k}\circ\cdots\circ\widehat{A_1}) (\Z^n)\right)\]

Then, Theorem~\ref{MinkAlm} applies to prove next result.

\begin{theoreme}\label{AnswerConjIsom}
Let $(P_k)_{k\ge 1}$ be a generic\footnote{\label{foot2}A property concerning elements of a topological set $X$ is called \emph{generic} if satisfied on at least a countable intersection of open and dense sets. In particular, Baire theorem implies that if this space is complete (as here), then this property is true on a dense subset of $X$.} sequence of matrices\footnote{The set of sequences of matrices is endowed with the norm $\|(P_k)_{k\ge 1}\| = \sup_{k\ge 1} \|P_k\|$, making it a complete space (see Note~\ref{foot2}).} of $O_n(\R)$. Then
\[\tau^k\big((P_k)_{k\ge 1}\big) \underset{k\to+\infty}{\longrightarrow} 0.\]
\end{theoreme}

Thus, for a generic sequence of angles, the application of successive discretizations of rotations of these angles to a numerical image will induce an arbitrarily large loss of quality of this image (see Figure~\ref{PoincareRot}).

\begin{figure}[t]
\begin{center}
\begin{minipage}[c]{.25\linewidth}
	\includegraphics[width=\linewidth]{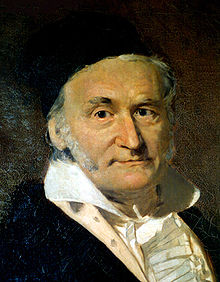}
\end{minipage}\hspace{50pt}
\begin{minipage}[c]{.25\linewidth}
	\includegraphics[width=\linewidth]{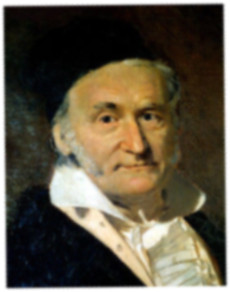}
\end{minipage}
\caption[40 successive rotations of an image]{Original image (left) of size $220\times 282$ and 10 successive random rotations of this image (right), obtained with the software \emph{Gimp} (linear interpolation algorithm).}\label{PoincareRot}
\end{center}	
\end{figure}

Let us sketch the proof of Theorem~\ref{AnswerConjIsom}. The idea is to study the set of differences of the sets
\[\Gamma_k = (\widehat{P_k}\circ\cdots\circ\widehat{P_1}) (\Z^n).\]
In particular, by analysing the action of the discretization of a generic map on the frequency of differences, one can prove the following lemma (we will admit its proof).

\begin{lemme}
For every $k$, for every isometry $P\in O_n(\R)$ and every $\varep>0$, there exists $\delta>0$ and a matrix $Q\in O_n(\R)$ such that $d(P,Q)<\varep$ satisfying: for every $v_0\in\Z^n$, 
\begin{enumerate}[(i)]
\item either there exists $v_1\in\Z^n\setminus\{0\}$ such that $\|v_1\|_2<\|v_0\|_2$ and that
\[\rho_{\widehat Q(\Gamma_k)}(v_1)\ge \delta \rho_{\Gamma_k}(v_0);\]
\item or
\[D(\widehat Q(\Gamma_k))\le D(\Gamma)\big(1-\delta \rho_{\Gamma_k}(v_0)\big).\]
\end{enumerate}
\end{lemme}

In other words, in case (i), making a $\varep$-small perturbation of $P$ if necessary, if a difference $v_0$ appears with a positive frequency in $\Gamma$, then some difference $v_1\neq 0$ will also appear with positive frequency, with the fundamental property that $\|v_1\|_2<\|v_0\|_2$. In case (ii), the rate of injectivity strictly decreases between times $k$ and $k+1$.

We then iterate this process, as long as we are in the first case of the lemma: starting from a difference $v_0$ appearing with a frequency $\rho_0$ in $\Gamma_k$, one can build a sequence of differences $(v_m)$ of vectors of $\Z^n$ with decreasing norm such that for every $m$ we have $\rho_{\Gamma_{k+m}}(v_m)\ge\delta^m\rho_0$. Ultimately, this sequence of points $(v_m)$ will go to 0 (as it is a sequence of integral points with decreasing norms). Thus, there will exist a rank $m_0\le \|v_0\|_2^2$ such that we will be in case (ii) of the lemma (which is the only case occurring when $\|v_0\|_2 = 1$). Then, we will get 
\[D(\Gamma_{k+m_0})\le D(\Gamma_k)\big(1-\delta^{m_0} \rho_{\Gamma}(v_0)\big).\]

It remains to initialize this construction, that is, to find a difference $v_0\in\Z^n$ ``not too far from 0'' and such that $\rho_{\Gamma}(v_0)$ is large enough. This step simply consists in the application of Theorem~\ref{MinkAlm}: applying it to $S=B(0,r)$ with $r^2=8/(\pi D(\Gamma))$, one gets
\[\sum_{u\in B(0,r)} \rho_\Gamma(u) \ge 2,\]
thus
\[\sum_{u\in B(0,r)\setminus\{0\}} \rho_\Gamma(u) \ge 1.\]
As the support of $\rho_\Gamma$ is included in $\Z^n$, and as $\card (B(0,r)\cap\Z^n)\le \pi(r+1)^2$, this implies that there exists $u_0\in B(0,r)\cap (\Z^n\setminus \{0\})$ such that
\[\rho_\Gamma(u_0) \ge \frac{1}{\pi(r+1)^2},\]
which gives for $r\ge 3$
\[\rho_\Gamma(u_0) \ge \frac{D(\Gamma)}{16}.\]

This allows to estimate the ``loss of injectivity'' $D(\Gamma_{k}) - D(\Gamma_{k+m_0})$ that occurs between times $k$ and $k+m_0$. Theorem~\ref{AnswerConjIsom} is obtained by applying this reasoning many times.

\subsubsection*{Acknowledgements.}

The first author is founded by an IMPA/CAPES grant. The second author is funded by the French Agence Nationale de la Recherche (ANR), under grant ANR-13-BS01-0005 (project SPADRO).

\bibliographystyle{amsalpha}
\bibliography{Biblio}

\end{document}